\documentclass[12pt]{article}
\usepackage{}
\usepackage{amssymb}
\usepackage{amsfonts}
\usepackage{amsmath}
\usepackage[usenames]{color}
\usepackage{mathrsfs}
\usepackage{amsfonts}
\usepackage{amssymb,amsmath}
\usepackage{CJK}
\usepackage{cite}
\usepackage{cases}
\usepackage{amsthm}

\def\cl{\centerline}

\def\vs{\vspace*}
\def\H{\mathfrak{a}}
\def\L{\mathcal{L}}

\def\B{\mathfrak{b}}
\def\Z{\mathbb{Z}}

\def\C{\mathbb{C}}

\numberwithin{equation}{section}
\newtheorem{th*}{Theorem}
\newtheorem{theo}{Theorem}[section]

\newtheorem{coro}[theo]{Corollary}
\newtheorem{lemm}[theo]{Lemma}

\newtheorem{prop}[theo]{Proposition}

\newtheorem{case}{Case}

\newtheorem{remark}[theo]{Remark}

\begin{document}
\begin{center}
{\large\bf A class of  non-weight modules  over the Virasoro algebra}
\end{center}
\vs{5pt}
\cl{\small Haibo Chen$^{\rm a}$,  Jianzhi Han$^{{\rm b},\, *}$}\footnote{$^{*}$ Corresponding author.  E-mail address: jzhan@tongji.edu.cn (J. Han).}\footnote{$^*$ Partially supported by NSF of China (Grant 11501417)
and the Innovation Program of Shanghai Municipal Education Commission.}
\cl{$^{\rm a}$ \small  School of  Statistics and Mathematics, Shanghai Lixin University of } \cl{\small Accounting and Finance,   Shanghai
201209, China}
\cl{$^{\rm b}$ \small  School of Mathematical Sciences, Tongji University, Shanghai
200092, China}

\vs{8pt}
{\small
\parskip .005 truein
\baselineskip 3pt \lineskip 3pt
\noindent{{\bf Abstract:}  For any  triple $(\mu,\lambda,\alpha)$ of complex numbers  and  an $\mathfrak a$-module  ${V}$,   a class of  non-weight modules $\mathcal{M}\big(V,\mu,\Omega(\lambda,\alpha)\big)$ over
the Virasoro algebra $\mathcal L$ is constructed in this paper.
 We prove if $V$ is a nontrivial simple $\mathfrak a$-module satisfying: for any $v\in V$ there exists $r\in\Z_+$ such  that $L_{r+i}v=0$ for all $i\geq1$, then $\mathcal{M}\big(V,\mu,\Omega(\lambda,\alpha)\big)$ is simple if and only if $\mu\neq1, \lambda\neq0,\alpha\neq0,$.
We also give the necessary and sufficient conditions for two such simple $\mathcal L$-modules being isomorphic. Finally,     we prove that these simple $\mathcal L$-modules $\mathcal{M}\big(V,\mu,\Omega(\lambda,\alpha)\big)$ are new
 by  showing they are not isomorphic to any  other known simple non-weight module provided that $V$ is not a  highest weight $\mathfrak a$-module with  highest weight nonzero.

\vs{5pt}

\noindent{\bf Key words:}
Virasoro algebra,  non-weight module,   simple module.}

\noindent{\it Mathematics Subject Classification (2010):} 17B10, 17B65, 17B68.}
\parskip .001 truein\baselineskip 6pt \lineskip 6pt

\section{Introduction}
The  {\em  Virasoro algebra} $\L$     is
 an infinite dimensional complex Lie algebra
with the basis  $\{L_m:=t^{m+1}\frac{d}{dt},C
\mid m\in \Z\}$ and  the following Lie brackets:
\begin{equation*}
[t^{m+1}\frac{d}{dt},t^{n+1}\frac{d}{dt}]= (n-m)t^{m+n+1}\frac{d}{dt}+\delta_{m+n,0}\frac{m^{3}-m}{12}C\ {\rm and}\ [t^{m+1}\frac{d}{dt},C]=0\ {\rm for}\ m,n\in\Z,
\end{equation*}  which  is known as   one of the most  important infinite dimensional Lie algebras   both in mathematics and mathematical  physics.
 The  theory of  weight modules over $\L$ is well developed (see \cite{IK}). A weight $\L$-module whose weight subspaces are all finite dimensional is called a {\em Harish-Chandra module}. In fact, any simple  weight module over the Virasoro algebra with  a nonzero finite dimensional weight space is a  Harish-Chandra module (see \cite{MZ1}). And the classification of simple Harish-Chandra modules over $\L$ was already achieved  (see \cite{M,S,MP}).
After that,  the study of weight modules turns to  modules  with an infinite dimensional weight space (see, e.g., \cite{CM,LZ0,LLZ}).
Such modules were  first constructed by taking the tensor product of some highest weight modules and some intermediate
series modules (see \cite{Z}), whose simplicities  were determined  in \cite{CGZ}.

 Non-weight $\L$-modules, as  the other component of  representation theory of the Virasoro algebra,
have drawn much attention in the past few years, such as Whittaker modules (see, e.g., \cite{LZ2,OW,LGZ,MW,MZ}),
 $\C[L_0]$-free modules,
simple module from Weyl modules and a class of non-weight modules  including highest-weight-like modules
 (see, e.g, \cite{LZ,TZ,TZ1,LZ2,GLZ,CG}). In the present paper, we shall study non-weight $\L$-modules. To be more precisely, we are going to construct a family of new simple $\L$-modules from the tensor products of $\L$-modules  $\Omega(\lambda,\alpha)=\C[\partial]$ (see \cite{LZ2}) and $\H$-modules, where $\H=\mathrm{span}\{L_i\mid i\geq-1\}.$

Here follows a brief summary of this  paper.
 In Section $2$, we construct  a class of non-weight   modules $\mathcal{M}\big(V,\mu,\Omega(\lambda,\alpha)\big)$ associated to $\H$-modules ${V}$ and $\L$-modules $\Omega(\lambda,\alpha)$. The simplicities of modules in this class are determined in
  Section $3$.  We show that $\mathcal{M}\big(V,\mu,\Omega(\lambda, $ $\alpha)\big)$ is simple if and only if $\mu\neq1$ and $\alpha\neq0$.
 Section 4 is devoted to   giving the necessary and sufficient conditions for two   simple $\L$-modules $\mathcal{M}\big(V_1,\mu_1,\Omega(\lambda_1,\alpha_1)\big)$
 and  $\mathcal{M}\big(V_2,\mu_2,\Omega(\lambda_2,\alpha_2)\big)$  being isomorphic.    In Section 5,   we compare   simple $\L$-modules  constructed in the present paper    with the known simple non-weight $\L$-modules and  show that all simple $\L$-modules $\mathcal{M}\big(V,\mu,\Omega(\lambda,\alpha)\big)$ ($V$ being not a highest weight module with highest weight in $\C^*$)   are new.

Throughout this paper, we respectively denote by $\C,\C^*,\Z$, $\Z_+$ and $\mathcal U(\mathfrak g)$   the sets of complex numbers, nonzero complex numbers, integers, nonnegative integers and the enveloping algebra of a Lie algebra $\mathfrak g$.
 All vector spaces are assumed to be over $\C$.

\section{Non-weight modules $\mathcal{M}\big(V,\mu,\Omega(\lambda,\alpha)\big)$}
%$In this section, we shall construct a class of    non-weight modules over the Virasoro algebra.

%For any $\H$-module  ${V}$ and $v\in{V}$, define the {\em order} of $v$, denoted by $\mathrm{ord}_{\H}(v)$, as the minimal nonnegative integer $r$ with $L_{r+i}v=0$ for all $i\geq1$, or as $\infty$ if this $r$ does not exist. Furthermore, the order of ${V}$, denoted by $\mathrm{ord}_{\H}({V})$, is defined as the maximal order of all its elements or $\infty$ if it does not exist.

Recall from \cite{LZ2}  that for $\lambda\in\C,\alpha\in\C$  the non-weight $\L$-module $\Omega(\lambda,\alpha)=\C[\partial]$,
whose  actions are given  by
$$L_mf(\partial)=\lambda^m(\partial-m\alpha)f(\partial-m)\quad {\rm and}\quad Cf(\partial)=0\quad {\rm  for}\ m\in\Z, f(\partial)\in\C[\partial].$$
 It is worthwhile to point out that $\Omega(\lambda,\alpha)$ is simple
 if and only if $\lambda\neq0,\alpha \neq0$ (see \cite{LZ2}).

Let  $\C[[t]]$ be the algebra of formal power series.
 Denote
 $$e^{mt}=\sum_{i=0}^\infty\frac{(mt)^i}{i!}\in\C[[t]]\quad {\rm for}\ m\in\Z.$$ Then $e^{mt}\frac{d}{dt}=\sum_{i=0}^\infty\frac{m^i}{i!}L_{i-1}$
and
 \begin{equation}\label{emt2.1}
 [e^{mt}\frac{d}{dt},e^{nt}\frac{d}{dt}]=(n-m)e^{(m+n)t}\frac{d}{dt}\quad {\rm for}\ m,n\in\Z.
 \end{equation} So  $\C[[t]]\frac{d}{dt}$
 carries the structure of a Lie algebra.

 Consider the following two subalgebras of $\L:$ $$\H=\mathrm{span}\{L_i\mid i\geq-1\}\quad{\rm and}\quad \B=\mathrm{span}\{L_i\mid i\geq0\}. $$
Let $V$ be an $\H$-module such that for any $v\in V$, $L_mv=0$ for all  but  finitely many of $m (\geq-1)$.
For any $\mu,\lambda, \alpha\in\C$, inspired by \cite{LZ0} we define an $\L$-action  on the vector space $\mathcal{M}\big(V,\mu,\Omega(\lambda,\alpha)\big):=V\otimes \C [\partial]$ as follows\vspace{-0.3cm}
 \begin{eqnarray}
 &\label{Lm2.1}  L_m\big(v\otimes f(\partial)\big)=v\otimes\lambda^m(\partial-m\alpha)f(\partial-m)+(\mu^me^{mt}-1)\frac{d}{dt}v\otimes \lambda^mf(\partial-m),\\&
\label{C1232.3}  C\big(v\otimes f(\partial)\big)=0\quad {\rm for}\ m\in\Z,v\in V, f(\partial)\in\C[\partial].
 \end{eqnarray}

\begin{prop}
Let $\mu,\lambda,\alpha\in\C$ and $V$ be    an $\H$-module as above. Then
 $\mathcal{M}\big(V,\mu,\Omega(\lambda,\alpha)\big)$
is a non-weight $\L$-module under the actions given in \eqref{Lm2.1} and \eqref{C1232.3}.
\end{prop}
\begin{proof} Define a series of operators $x_m$ on  $\C[\partial]$ as follows:  $$x_mf(\partial)=\lambda^mf(\partial-m) \quad{\rm for}\  m\in\Z\ {\rm and}\  f(\partial)\in\C[\partial].$$ Then $x_nx_mf(\partial)=x_{m+n}f(\partial)$ for   $m,n\in\Z$.
It follows from  \eqref{Lm2.1} and \eqref{C1232.3} that we have
 \begin{eqnarray*}
&&(L_mL_n-L_nL_m)\big(v\otimes f(\partial)\big)=L_m\big(v\otimes  (\partial-n\alpha)x_nf(\partial)+(\mu^ne^{nt}-1)\frac{d}{dt}v\otimes x_nf(\partial)\big)\\
&&-L_n\big(v\otimes   (\partial-m \alpha)x_mf(\partial)+(\mu^me^{mt}-1)\frac{d}{dt}v\otimes x_mf(\partial)\big)\\
&=&
v\otimes  (\partial-m\alpha)(\partial-n\alpha-m)x_{m+n}f(\partial)+(\mu^me^{mt}-1)\frac{d}{dt}v\otimes (\partial-n\alpha-m)x_{m+n}f(\partial)
\\&&
+\ (\mu^ne^{nt}-1)\frac{d}{dt}v\otimes (\partial-m\alpha)x_{m+n}f(\partial)
+((\mu^me^{mt}-1)\frac{d}{dt})(\mu^ne^{nt}-1)\frac{d}{dt}v\otimes x_{m+n}f(\partial)
\\&&
-\ v\otimes  (\partial-n\alpha)(\partial-m\alpha-n)x_{m+n}f(\partial)
-\ (\mu^ne^{nt}-1)\frac{d}{dt}v\otimes(\partial-m\alpha-n) x_{m+n}f(\partial)
\\&&
-(\mu^me^{mt}-1)\frac{d}{dt}v\otimes (\partial-n\alpha)x_{m+n}f(\partial)
-\big((\mu^ne^{nt}-1)\frac{d}{dt}\big)(\mu^me^{mt}-1)\frac{d}{dt}v\otimes x_{m+n}f(\partial)
%\\&=&
%(n-m)v\otimes  (t-(m+n)\alpha)x_{m+n}f(\partial)
%\\&&
%-m\big((\mu^me^{mt}-1)\frac{d}{dt}\big)v\otimes x_{m+n}f(\partial)
%+n\big((\mu^ne^{nt}-1)\frac{d}{dt}\big)v\otimes x_{m+n}f(\partial)
%\\&&
%+\big((\mu^me^{mt}-1)\frac{d}{dt}\big)\big((\mu^ne^{nt}-1)\frac{d}{dt}\big)v\otimes x_{m+n}f(\partial)
%\\&&
%-\big((\mu^ne^{nt}-1)\frac{d}{dt}\big)\big((\mu^me^{mt}-1)\frac{d}{dt}\big)v\otimes x_{m+n}f(\partial)
\\&=&
(n-m)v\otimes  \big(\partial-(m+n)\alpha\big)x_{m+n}f(\partial)
 +(n-m)(\mu^{m+n}e^{(m+n)t}-1)\frac{d}{dt}v\otimes x_{m+n}f(\partial)
\\&=&(n-m)L_{m+n}\big(v\otimes f(\partial)\big).
\end{eqnarray*}
Thus,  $\mathcal{M}\big(V,\mu,\Omega(\lambda,\alpha)\big)$ becomes a non-weight $\L$-module.
\end{proof}

A module $V$ over a Lie algebra $\mathfrak g$ is called {\it trivial} if $xV=0$ for any $x\in\mathfrak g,$ and {\it nontrivial} otherwise.
\begin{remark}\label{remark-han-add}\rm
\begin{itemize}\parskip-3pt
\item[(1)] Let  $M$ be a  $\B$-module for which there exists $r\in\Z_+$ such that $L_{r+i}M=0$ for all $i\geq 1$.  Then  $\mathcal{M}\big(M,1,\Omega(\lambda,\alpha)\big)$ is, in fact,
 the module  $F\big(M,\Omega(\lambda,\alpha)\big)$ studied in \cite{LZ}.

 \item[(2)] If $V$ is a trivial $\H$-module, then $\mathcal{M}\big(V,\mu,\Omega(\lambda,\alpha)\big)\cong \Omega(\lambda,\alpha)$ as $\L$-modules.
 \end{itemize}
\end{remark}

%The isomorphism classes of $\L$-module $\mathcal{M}\big(V,\mu,A(\lambda,\alpha)\big)$ is given as follows.
%\begin{theo}{\rm \cite{LZ0}}\label{th2.3}\quad
%Let $\mu_i\in\C^*$, $\lambda_i, \alpha_i\in\C$ with $0\leq\mathrm{Re}(\lambda_i)<1$ and $V_i$ be a nontrivial simple $\H$-module in $\mathcal C$ for $i=1,2$. Suppose that $\mathcal{M}\big(V_1,\mu_1,A(\lambda_1,\alpha_1)\big)$ is simple.
%Then   $\mathcal{M}\big(V_1,\mu_1,A(\lambda_1,\alpha_1)\big)$ and  $\mathcal{M}\big(V_2,\mu_2,A(\lambda_2,\alpha_2)\big)$  are isomorphic as $\L$-modules if and only  if
%one of the following holds
 %\begin{itemize}\parskip-3pt
%\item[{\rm (1)}]
%$V_1\cong V_2,\mu_1=\mu_2, \lambda_1=\lambda_2$ and $\alpha_1=\alpha_2$, or
%\item[{\rm (2)}]
%$V_1,V_2$ are the highest weight modules with the nonzero highest weights $\beta_1,\beta_2$, respectively, and
% $\mu_2=\mu_1^{-1},\lambda_1=\lambda_2,\alpha_1=\beta_2+1,\alpha_2=\beta_1+1$.
%\end{itemize}
%\end{theo}

\section{Simplicity   of   $\mathcal{M}\big(V,\mu,\Omega(\lambda,\alpha)\big)$}
%The aim  of this section is to  give a  characterisation of  the simplicity of $\mathcal{M}\big(V,\mu,\Omega(\lambda,\alpha)\big)$.

From now on, {\bf let $\mathcal C$ denote the category of all nontrivial $\H$-modules $V$ satisfying the following condition:} %for any  $0\neq v\in V$  there exists $r\in\Z_+$ such that $L_{r+i}v=0$ { for all $i\geq1.$}
\begin{eqnarray*}&{\bf for\ any}\  0\neq v\in V\ {\bf there\ exists}\ r_v\in\Z_+\ {\bf such\ that}\ L_{r_v+i}v=0\ {\bf for\ all}\ i\geq 1.  \end{eqnarray*}
The minimal such $r_v$ is called {\it the order of} $v,$ denoted by ${\rm ord}(v)$.  Choose $r\in\Z_+$ minimal such that the set $V_\B=\{v\in V\mid L_{r+i}v=0\ {\rm for\ all}\ i\geq1\}$ is nonzero, which is denoted by ${\rm ord}(V_\B)$.
The following lemma   can be found in \cite[Lemmas 1 and 2]{LZ0}.
\begin{lemm}\label{lemm3.1}
 Suppose that $V$ and $W$ are  simple modules in $\mathcal C.$ Then
 \begin{itemize}\parskip-7pt
 \item[{\rm (1)}]$V_\B$  is a simple $\B$-module$;$
 \item[{\rm (2)}] $V\cong \mathcal U(\H)\otimes_{\mathcal U(\B)}V_\B\ (=\C[L_{-1}]\otimes V_\B\ {\rm linearly})$ as $\H$-modules$;$
\item[{\rm (3)}]  $L_r$ acts bijectively on $V_\B$ and ${\rm ord}(f(L_{-1})v)={\rm deg}\, f+r$ if $V_\B$ is a nontrivial $\B$-module$,$ where $r={\rm ord}(V_\B), v\in V_{\B}$ and $0\neq f(x)\in\C[x];$
\item[{\rm(4)}] $V\cong W$ as $\H$-modules if and only if $V_\B\cong W_\B$ as $\B$-modules.
\end{itemize}
\end{lemm}

For any $m\in\Z,n\in\Z_+$, we denote
$$J_m^0=1\ \mbox{and}\ J_m^n=\prod_{j=m+1}^{m+n}(\partial-j)\quad \mbox{for}\ n>0.$$
Note that $\{J_m^n\mid n\in\Z_+\}$ forms a basis of $\Omega(\lambda,\alpha)$  for any $m\in\Z$.
By the action  of $\L$ on  $\Omega(\lambda,\alpha)$, it is easy to check that
\begin{equation*}%\label{lmj3.1}
L_mJ_n^k=\lambda^m(\partial-m\alpha)J_{m+n}^{k}\quad {\rm for}\ m,n\in\Z,k\in\Z_+.
\end{equation*}

Let $M$ be a  $\B$-module.  Denote  $M^{(p)}=\sum_{i=0}^pL_{-1}^iM\otimes \C[\partial]$ for any $p\in\Z_+$. %?????If in addition $M$ is simple and  $\mathrm{ord}_{\B}(M)=r\geq0,$  then one can see that $L_r$ acts bijectively   on $M$ (cf. \cite{LLZ}).

%Now we prove the following lemma, which determines the simplicity of $\mathcal{M}\big(V,\mu,\Omega(\lambda,\alpha)\big)$ for $\mu=1$.

\begin{lemm}\label{lemm3.2}
Let $\lambda\in\C^*,\alpha\in\C$ and   $V$ be an $\H$-module in $\mathcal C$.   Then  $\mathcal{M}\big(V,1,\Omega(\lambda,\alpha)\big)$
has a series of $\L$-submodules
$$V_{\B}^{(0)}\subset V_\B^{(1)}\subset\cdots V_\B^{(p)}\subset\cdots$$ such that $V^{(n)}/V^{(n-1)}\cong F\big(V_{\B},\Omega(\lambda,\alpha)\big)$  as $\L$-modules for each $n\geq1$.

\end{lemm}
\begin{proof}
Following from \eqref{Lm2.1} it is easy to check that $V^{(0)}$ is an $\L$-submodule of $\mathcal{M}\big(V,1,\Omega(\lambda,$ $\alpha)\big).$ Suppose that $V_\B^{(0)}, V_\B^{(1)},\ldots, V_\B^{(n-1)}$ are all $\L$-submodules. Take $L_{-1}^n v\otimes J_0^k+v^{(n-1)}\in V_\B^{(n)},$ where  $v\in V, k\in\Z_+$ and $v^{(n-1)}\in V_\B^{(n-1)}.$ Then by inductive assumption,
\begin{eqnarray}\label{han-xx}
&&L_m(L_{-1}^{n}v\otimes J_0^k+v^{(n-1)})\nonumber\\
&\equiv&L_{-1}^{n}v\otimes\lambda^m(\partial-m\alpha)J_m^k+\big((e^{mt}-1)\frac{d}{dt}\big)L_{-1}^{n}v\otimes \lambda^mJ_m^k\quad({\rm mod}\ V^{(n-1)})
\\
&\equiv&L_{-1}^{n}v\otimes\lambda^m(\partial-m\alpha)J_m^k+(L_{-1}-m)^{n}\big(e^{mt}\frac{d}{dt}\big)v\otimes \lambda^mJ_m^k \nonumber\\
&&-L_{-1}^{n+1}v\otimes \lambda^mJ_m^k \quad({\rm mod}\ V^{(n-1)})\nonumber
\\ &\equiv&L_{-1}^{n}v\otimes\lambda^m(\partial-m\alpha)J_m^k+L_{-1}^n(\sum_{i\geq1}\frac{m^i}{i!}L_{i-1}-mn)v\otimes \lambda^mJ_m^k \quad({\rm mod}\ V^{(n-1)})\nonumber
\\
&\equiv&0\nonumber
\quad(\mathrm{mod}\ V_\B^{(n)}),
\end{eqnarray} where of course we use the formula:
\begin{eqnarray}
\label{eq-extra}\big(e^{kt}\frac{d}{dt}\big)L_{-1}^{i}=\big(L_{-1}-k\big)^{i}e^{kt}\frac{d}{dt}\quad {\rm for}\ i\in\Z_+.
\end{eqnarray}
This shows that each $V_\B^{(n)}$ is a submodule of $\mathcal{M}\big(V,1,\Omega(\lambda,\alpha)\big).$ And the $\L$-module isomorphism $$V^{(n)}/V^{(n-1)}\cong F\big(V_{\B},\Omega(\lambda,\alpha)\big)$$  follows immediately from \eqref{han-xx} and Remark \ref{remark-han-add}.
\end{proof}

\begin{lemm}\label{lemm3.2222}
Let   $V$ be an $\H$-module. Then the subspace
$$\tilde{\mathcal{M}}\big(V,\mu,\Omega(\lambda,0)\big)=\{L_{-1}v\otimes J_0^k-v\otimes \partial J_0^k\mid k\in\Z_+,v\in V \}$$
is an $\L$-submodule of $\mathcal{M}\big(V,\mu,\Omega(\lambda,0)\big)$ isomorphic to $\mathcal{M}\big(V,\mu,\Omega(\lambda,1)\big).$ Moreover$,$ $$\mathcal{M}\big(V,\mu,\Omega(\lambda,0)\big)/\tilde{\mathcal{M}}\big(V,\mu,\Omega(\lambda,0)\big)\cong F(V,\Omega(\lambda\mu,0))\ {\it as}\ \L\text-{\it modules.}$$
\end{lemm}
\begin{proof}Without loss of generality, we assume that $\lambda=1$. Note that for any $n\in\Z,k\in\Z_+$ and $v\in V$ that
\begin{eqnarray*}
&&L_n(L_{-1}v\otimes J_0^k-v\otimes \partial J_0^k)
\\&=&L_{-1}v\otimes \partial J_{n}^k+(\mu^ne^{nt}-1)\frac{d}{dt}L_{-1}v\otimes  J_{n}^k-
\\&&v\otimes\partial (\partial-n) J_{n}^k-(\mu^ne^{nt}-1)\frac{d}{dt}v\otimes (\partial-n) J_{n}^k\\
&=&L_{-1}v\otimes \partial J_{n}^k+(L_{-1}-n)\mu^ne^{nt}\frac{d}{dt}v\otimes  J_{n}^k-L^2_{-1}v\otimes  J_{n}^k-
\\&&v\otimes\partial (\partial-n) J_{n}^k-(\mu^ne^{nt}-1)\frac{d}{dt}v\otimes \partial J_{n}^k+n(\mu^ne^{nt}-1)\frac{d}{dt}v\otimes  J_{n}^k
\\&=&L_{-1}v\otimes (\partial-n) J_{n}^k-v\otimes\partial (\partial-n) J_{n}^k+
\\&& L_{-1}(\mu^ne^{nt}-1)\frac{d}{dt}v\otimes  J_{n}^k
-(\mu^ne^{nt}-1)\frac{d}{dt}v\otimes \partial J_{n}^k\in \tilde{\mathcal{M}}\big(V,\mu,\Omega(\lambda,0)\big).
\end{eqnarray*}
So  the subspace $\tilde{\mathcal{M}}\big(V,\mu,\Omega(\lambda,0)\big)$ does form an $\L$-submodule of ${\mathcal{M}}\big(V,\mu,\Omega(\lambda,0)\big)$.

Let $\tau:\mathcal{M}\big(V,\mu,\Omega(\lambda,1)\big)\rightarrow\tilde{\mathcal{M}}\big(V,\mu,\Omega(\lambda,0)\big)$ be a linear isomorphism given by
\begin{eqnarray*}
\tau(v\otimes J_0^k)=L_{-1}v\otimes J_0^k-v\otimes \partial J_0^k\quad {\rm for}\ k\in\Z_+,v\in\Z.
\end{eqnarray*}
Then a direct computation:
\begin{eqnarray*}
&&\!\!\!\!\!\!\!\!\tau(L_n(v\otimes J_0^k))=\tau(v\otimes (\partial-n) J_{n}^k
+(\mu^ne^{nt}-1)\frac{d}{dt}v\otimes J_{n}^k)
\\=&&\!\!\!\!\!\!\!\!L_{-1}v\otimes (\partial-n) J_{n}^k-v\otimes \partial(\partial-n) J_{n}^k+L_{-1}(\mu^ne^{nt}-1)\frac{d}{dt}v\otimes J_{n}^k-(\mu^ne^{nt}-1)\frac{d}{dt}v\otimes \partial J_{n}^k
\\=&&\!\!\!\!\!\!\!\!L_n\tau(v\otimes J_0^k)
\end{eqnarray*}
shows that $\tau$ is a homomorphism and therefore an isomorphism of $\L$-modules.
Observe that
\begin{eqnarray*}
&&L_n(v\otimes J_0^k)=v\otimes\partial J_{n}^k
+(\mu^ne^{nt}-1)\frac{d}{dt}v\otimes J_{n}^k
\\&=&v\otimes\mu^n \partial J_{n}^k
+(e^{nt}-1)\frac{d}{dt}v\otimes\mu^n J_{m+n}^k +(\mu^n-1)\Big(L_{-1}v\otimes J_{n}^k-v\otimes \partial J_{n}^k\Big)
\\&\equiv&v\otimes\mu^n \partial J_{n}^k
+(e^{nt}-1)\frac{d}{dt}v\otimes\mu^n J_{n}^k\quad(\mathrm{mod}\ \tilde{\mathcal{M}}\big(V,\mu,\Omega(\lambda,0)\big)),
\end{eqnarray*}
proving the second statement and completing the proof.
\end{proof}

 The following result  is useful.
\begin{prop}{\rm \cite{LZ0}}\label{pro13.3}
Let $P$ be a vector space over $\C$ and $P_1$  a subspace of $P$. Suppose
that $\mu_1,\mu_2,\ldots,\mu_s\in\C^*$ are pairwise distinct$,$ $v_{i,j}\in P$
and $f_{i,j}(t)\in\C[t]$ with $\mathrm{deg}\,f_{i,j}(t)=j$ for $i=1,2,\ldots,s;j=0,1,2,\ldots,k.$
If $$\sum_{i=1}^{s}\sum_{j=0}^{k}\mu_i^mf_{i,j}(m)v_{i,j}\in P_1\quad{ \it for}\ K< m\in\Z\ (K \ {\it any\ fixed\ element\ in}\  \Z\cup\{-\infty\})$$
then $v_{i,j}\in P_1$ for all $i,j$.
\end{prop}

\begin{remark}\rm
 Though  the statement of Proposition \ref{pro13.3} is slightly different from \cite[Proposition 7]{LZ0} (in which $K$ is only taken to be  $-\infty),$   it follows from the proof there  that Proposition \ref{pro13.3} also holds.
\end{remark}

\begin{lemm}\label{lemm3.55}
 Let $\lambda,\alpha, 1\neq\mu\in\C^*,$ $V$ be a  simple $\H$-module in $\mathcal C$ and    $W$   an  $\L$-submodule of $\mathcal{M}\big(V,\mu,\Omega(\lambda,\alpha)\big)$. Suppose  $0\neq u\otimes   f(\partial) \in W$ for some $u\in V$ and $f(\partial)\in\C[\partial]$.
  Then  $W=\mathcal{M}\big(V,\mu,\Omega(\lambda,\alpha)\big).$

\end{lemm}
\begin{proof}   Let $r\geq-1$ be the maximal integer such that  $L_ru\neq0$.  Note for any $m\in\Z$ that
\begin{eqnarray*}\label{WL3.2}
W&\ni& L_m\big(u\otimes f(\partial)\big)\nonumber\\
&=&\lambda^m \Big\{u\otimes(\partial-m\alpha) f(\partial-m)+\big((\sum_{j=0}^{r+1}\frac{\mu^mm^{j}}{j!} L_{j-1}-L_{-1})u\big)\otimes f(\partial-m)\Big\}.
\end{eqnarray*}
Applying  Proposition  \ref{pro13.3} here     one has
$ 0\neq (L_{r}u)\otimes  1\in W$. Then by using $L_0^i(v\otimes 1)=v\otimes \partial^i$ for all $v\in V$ and $i\in\Z_+$ we see that $(L_ru)\otimes \Omega(\lambda,\alpha)\subseteq W$. Set $M=\{w\in V\mid w\otimes \Omega(\lambda,\alpha)\subseteq W\}$.   Replacing $u$ by $L_ru$ in the above procedure and using Proposition \ref{pro13.3} again give $(L_iw)\otimes \Omega(\lambda,\alpha) \subseteq W$ for any $i\geq-1$. In particular,   since $V$ is simple, $M=V$ and $W=\mathcal{M}\big(V,\mu,\Omega(\lambda,\alpha)\big)$.
\end{proof}

 Now we are ready to state the  first main result of this paper.
\begin{theo}\label{th1}
Let $\mu,\lambda\in\C^\ast,\alpha\in\C$ and $V$  be a  simple $\H$-module in $\mathcal C.$ Then the $\L$-module $\mathcal{M}\big(V,\mu,\Omega($ $\lambda,\alpha)\big)$ is simple if and only if $\alpha\neq0$ and $\mu\neq 1$.
\end{theo}
\begin{proof}Note by Lemma \ref{lemm3.2} and Lemma \ref{lemm3.2222} that  $\mathcal{M}\big(V,1,\Omega(\lambda,\alpha)\big)$ and $\mathcal{M}\big(V,\mu,\Omega(\lambda,0)\big)$ are not simple. Consider now $\alpha\neq0$ and $\mu\neq1$.
Let $W$ be a nonzero submodule of  $\mathcal{M}\big(V,\mu,\Omega(\lambda,\alpha)\big)$.
Without loss of generality, we may
assume that $\lambda=1$.

Take a nonzero element $u=\sum_{i=0}^{p}a_iL_{-1}^{i}v_i\otimes   J_{0}^{n_i}\in W\subseteq \C[L_{-1}] V_\B\otimes \C[\partial]$ with  $a_pL_{-1}^pv_p\otimes J_{0}^{n_p}\neq0.$
Then by \eqref{eq-extra},  one can easily compute that
\begin{eqnarray*}\label{3.4}
&&L_kL_{m-k}(u)=\sum_{i=0}^p a_iL_kL_{m-k}(L_{-1}^{i}v_i\otimes J_0^{n_i} )   \nonumber
\\&=&\sum_{i=0}^pa_i\Big\{L_{-1}^{i}v_i\otimes(\partial-k\alpha)(\partial-(m-k)\alpha-k)J_m^{n_i}+    \nonumber
\\&&\big((L_{-1}-k\big)^{i}\mu^{k}e^{kt}\frac{d}{dt}-L_{-1}^{i+1}\big)v_i\otimes (\partial-(m-k)\alpha-k)J_{m}^{n_i}+   \nonumber
\\&&\big((L_{-1}-m+k\big)^{i}\mu^{(m-k)}e^{(m-k)t}\frac{d}{dt}-L_{-1}^{i+1}\big)v_i\otimes (\partial-k\alpha)J_{m}^{n_i} + \nonumber
\\&&\big(L_{-1}-m\big)^{i}\big(\mu^{k}e^{kt}\frac{d}{dt}\big)\big(\mu^{(m-k)}e^{(m-k)t}\frac{d}{dt}\big)v_i\otimes J_m^{n_i}
-\big(L_{-1}-k\big)^{i+1}\big(\mu^{k}e^{kt}\frac{d}{dt}\big)v_i\otimes J_m^{n_i}-
\\&&
L_{-1}\big(L_{-1}-m+k\big)^{i}\big(\mu^{(m-k)}e^{(m-k)t}\frac{d}{dt}\big)v_i\otimes J_m^{n_i}
+L_{-1}^{i+2}v_i\otimes J_m^{n_i}\Big\}\in W,
\end{eqnarray*}
which allows us to  write
\begin{eqnarray}\label{3.5}
&&L_kL_{m-k}(u)
=\sum_{j=0}^{r+p+2}\mu^{k}k^ju_{1,j}^{m,n}
+\sum_{j=0}^{r+p+2}\mu^{-k}k^jv_{\mu,j}^{m,n}
+\sum_{j=0}^{2r+2}k^jw_{\mu,j}^{m,n},
\end{eqnarray}
where $r\geq -1$ is the maximal integer such that $L_r$ is injective on $V_\B$,  $u_{1,j}^{m,n},v_{\mu,j}^{m,n},w_{\mu,j}^{m,n}\in W$ are independent of $k$ and
 \begin{eqnarray*}
&&\label{3.6}
 u_{1,r+p+2}^{m,n}=\frac{(-1)^pa_p\alpha }{(r+1)!}L_rv_p\otimes J_m^{n_p}.
\end{eqnarray*}
Applying Proposition \ref{pro13.3} to  \eqref{3.5} and by the choice of $r$ we can deduce from the  above expression that  $0\neq L_rv_p\otimes J_m^{n_p}\in W$. Then by Lemma \ref{lemm3.55}, $W=\mathcal{M}\big(V,\mu,\Omega(\lambda,\alpha)\big)$.
\end{proof}

\section{Isomorphism classes}
The second main result of this paper is  to give the isomorphisms between these modules of form $\mathcal{M}\big(V,\mu,\Omega(\lambda,\alpha)\big).$

\begin{lemm}\label{lemm4.1}
Let $\lambda,\alpha_1,\alpha_2,1\neq\mu\in\C^*$ and $V_1, V_2$  be  simple  highest weight $\H$-modules with   highest weights $-\alpha_2,-\alpha_1, $ respectively.     Then the liner map
\begin{eqnarray*}
\phi:\mathcal{M}\big(V_1,\mu,\Omega(\lambda,\alpha_1)\big)&\rightarrow&\mathcal{M}\big(V_2,\mu^{-1},\Omega(\mu\lambda,\alpha_2)\big),
\\ L_{-1}^iv_1\otimes f(\partial)&\mapsto&\sum_{p=0}^i(-1)^p\binom{i}{p}L_{-1}^pv_2\otimes \partial^{i-p}f(\partial)
\quad \mathit{for}\ f(\partial)\in\C[\partial],
\end{eqnarray*}
is an $\L$-module isomorphism$,$ where $v_i$ is the highest weight vector of $V_i$ for $i=1,2$.
\end{lemm}
\begin{proof}
  It suffices to show that $\phi$ is a homomorphism. Assume that $\lambda=1$. On one hand,
  \begin{eqnarray*}
&&L_n\phi(L_{-1}^iv_1\otimes J_{0}^k)=L_n\Big(\sum_{p=0}^i(-1)^p\binom{i}{p}L_{-1}^pv_2\otimes\partial^{i-p}J_{0}^k\Big)
\\ \!\!\!\!\!&=&\!\!\!\!\sum_{p=0}^i(-1)^p\binom{i}{p}\Big\{\mu^nL_{-1}^pv_2\otimes(\partial-n\alpha_2)(\partial-n)^{i-p}J_{n}^k+
 (e^{nt}-\mu^n)\frac{d}{dt}L_{-1}^pv_2\otimes(\partial-n)^{i-p}J_{n}^k\Big\}
\\&=&\sum_{p=0}^i(-1)^p\binom{i}{p}\bigg\{\mu^n\Big(L_{-1}^pv_2\otimes\partial(\partial-n)^{i-p}J_{n}^k-L_{-1}^{p+1}v_2\otimes(\partial-n)^{i-p}J_{n}^k-
\\&&~~~~~~~~~~~~~~~~~~~~~\ n\alpha_2L_{-1}^pv_2\otimes(\partial-n)^{i-p}J_{n}^k\Big)+(L_{-1}-n)^pL_{-1}v_2\otimes(\partial-n)^{i-p}J_{n}^k-
\\&&~~~~~~~~~~~~~~~~~~~~~n\alpha_1(L_{-1}-n)^pv_2\otimes(\partial-n)^{i-p}J_{n}^k\bigg\},
\end{eqnarray*}
and  on the other  hand,
\begin{eqnarray*}
&&\phi(L_n(L_{-1}^iv_1\otimes J_0^k))=\phi\big(L_{-1}^iv_1\otimes(\partial-n\alpha_1)J_{n}^k+(\mu^ne^{nt}-1)\frac{d}{dt}L_{-1}^iv_1\otimes J_{n}^k\big)
\\&=& \phi\big(L_{-1}^iv_1\otimes(\partial-n\alpha_1)J_{n}^k+\mu^n(L_{-1}-n)^i(L_{-1}-n\alpha_2)v_1\otimes J_{n}^k-L_{-1}^{i+1}v_1\otimes J_{n}^k\big)
\\&=& \phi\bigg\{\mu^n\Big((L_{-1}-n)^iL_{-1}v_1\otimes J_{n}^k-n\alpha_2(L_{-1}-n)^iv_1\otimes J_{n}^k\Big)
\\&&~~~+\ L_{-1}^iv_1\otimes\partial J_{n}^k-L_{-1}^{i+1}v_1\otimes J_{n}^k-n\alpha_1L_{-1}^iv_1\otimes J_{n}^k\bigg\}.
\end{eqnarray*}
So the following four formulas will be good enough to make  $\phi(L_n(L_{-1}^iv_1\otimes J_0^k))=L_n\phi(L_{-1}^iv_1\otimes J_0^k)$ hold:

\begin{eqnarray}\label{ln4.44}& \phi\big((L_{-1}-n)^iv_1\otimes J_{n}^k\big)=\sum\limits_{p=0}^i(-1)^p\binom{i}{p}L_{-1}^pv_2\otimes(\partial-n)^{i-p}J_{n}^k,
\\&\label{ln4.3}  \phi\big(L_{-1}^iv_1\otimes J_{n}^k\big)=\sum\limits_{p=0}^i(-1)^p\binom{i}{p}(L_{-1}-n)^pv_2\otimes(\partial-n)^{i-p}J_{n}^k,
\end{eqnarray}
\begin{eqnarray}
 &\!\!\!\!\!\!\phi((L_{-1}-n)^iL_{-1}v_1\otimes  J_{n}^k)=\sum\limits_{p=0}^i(-1)^p\binom{i}{p}( L_{-1}^pv_2\otimes\partial(\partial-n)^{i-p}J_{n}^k-L_{-1}^{p+1}v_2\otimes(\partial-n)^{i-p}J_{n}^k), \nonumber
\\ & \phi\big(L_{-1}^iv_1\otimes\partial J_{n}^k-L_{-1}^{i+1}v_1\otimes J_{n}^k\big)=\sum\limits_{p=0}^i(-1)^p\binom{i}{p}(L_{-1}-n)^pL_{-1}v_2\otimes(\partial-n)^{i-p}J_{n}^k. \nonumber
\end{eqnarray}
We only need to prove \eqref{ln4.44} and \eqref{ln4.3}, since the other two follow these ones.

Note that  $\mathcal{M}\big(V_2,\mu^{-1},\Omega(\mu\lambda,\alpha_2)\big)$ carries a natural $\C[L_{-1},\partial]$-module structure, since it is linearly isomorphic to the  algebra $\C[L_{-1},\partial]$. %Define a new action on   $\C[L_{-1},\partial]$ by $L_{-1}(L_{-1}^i\partial^j)=-L_{-1}^{i+1}\partial^j$ and $\partial(L_{-1}^i\partial^j)=L_{-1}^i\partial^{j+1}$ for any $i,j\in\Z_+.$
Then by this action we have
\begin{eqnarray*}\label{ln4.5}
\!\!\!\!\!\!\!\!&&\!\!\!\!\!\!\!\!\!\!\!\!\!\! \sum_{p=0}^i(-1)^p\binom{i}{p}\big((L_{-1}-n)^pv_2\otimes(\partial-n)^{i-p}J_{n}^k\big)=\sum_{p=0}^i \binom{i}{p}(n-L_{-1})^pv_2\otimes(\partial-n)^{i-p}J_{n}^k \nonumber
\\ =&&\!\!\!\!\!\!\!\!\big((n-L_{-1})+(\partial-n)\big)^i(v_2\otimes J_{n}^k)=(\partial-L_{-1})^i(v_2\otimes J_{n}^k)
\\=&&\!\!\!\!\!\!\!\!\sum_{p=0}^i(-1)^p \binom{i}{p}L_{-1}^pv_2\otimes\partial^{i-p}J_{n}^k =\phi\big(L_{-1}^iv_1\otimes J_{n}^k\big),
\end{eqnarray*} proving \eqref{ln4.3}. And \eqref{ln4.44} follows from a direct computation, completing the proof.
\end{proof}

%Now we consider
%\begin{eqnarray*}
%&&\sum_{p=0}^{i+1}(-1)^p{i+1\choose p}\Big((L_{-1}-n)^pv_2\otimes(\partial-n)^{i+1-p}\Big)
%\\&=&(-1)^{i+1}(L_{-1}-n)^{i+1}v_2\otimes1+\sum_{p=0}^{i}(-1)^p({i\choose p}+{i\choose p-1})\Big((L_{-1}-n)^pv_2\otimes(\partial-n)^{i+1-p}\Big)
%\\&=&(-1)^{i+1}(L_{-1}-n)^{i+1}v_2\otimes1+\sum_{p=0}^{i}(-1)^p{i\choose p}\Big((L_{-1}-n)^pv_2\otimes(\partial-n)^{i+1-p}\Big)+
%\\&&\sum_{p=1}^{i}(-1)^p{i\choose p-1}\Big((L_{-1}-n)^pv_2\otimes(\partial-n)^{i+1-p}\Big)
%\\&=&\phi\big(L_{-1}^iv_1\otimes\partial\big)-n\phi\big(L_{-1}^iv_1\otimes1\big)+
%\sum_{p=0}^{i}(-1)^{p+1}{i\choose p}\Big((L_{-1}-n)^pL_{-1}v_2\otimes(\partial-n)^{i-p}\Big)
%\\&&-n\sum_{p=0}^{i}(-1)^{p+1}{i\choose p}\Big((L_{-1}-n)^pv_2\otimes(\partial-n)^{i-p}\Big)
%\\&=&\sum_{p=0}^{i}(-1)^{p}{i\choose p}\Big((L_{-1}-n)^pv_2\otimes\partial(\partial-n)^{i-p}\Big)-
%\sum_{p=0}^{i}(-1)^{p}{i\choose p}\Big((L_{-1}-n)^pL_{-1}v_2\otimes(\partial-n)^{i-p}\Big)
%\\&=&\sum_{p=0}^{i}(-1)^{p}{i\choose p}\Big((L_{-1}-n)^pv_2\otimes(\partial-n)^{i+1-p}-
%(L_{-1}-n)^{p+1}v_2\otimes(\partial-n)^{i-p}\Big)
%\end{eqnarray*}

\begin{theo}\label{the4.4}
Let $ \lambda_i,\alpha_i,1\neq \mu_i\in\C^*$  and $V_i$ be a  simple $\H$-module in  $\mathcal C$  for $i=1,2$.
Then   $$\mathcal{M}\big(V_1,\mu_1,\Omega(\lambda_1,\alpha_1)\big)\cong \mathcal{M}\big(V_2,\mu_2,\Omega(\lambda_2,\alpha_2)\big)$$  as $\L$-modules if and only  if one of the following holds
\begin{itemize}\parskip-3pt
\item[{\rm (a)}] $(\mu_1, \lambda_1,\alpha_1)=(\mu_2, \lambda_2,\alpha_2)$  and $V_1\cong V_2$ as $\H$-modules$;$
\item[{\rm (b)}] $\mu_1=\mu_2^{-1}=\frac{\lambda_2}{\lambda_1}$  and  $V_{1}, V_{2}$  are highest weight $\H$-module\ with highest weights $-\alpha_2, -\alpha_1,$\ {\it respectively.}
\end{itemize}
\end{theo}
\begin{proof}
It is enough to show the ``only if" part, since the ``if" part follows from Lemma \ref{lemm4.1}.  Let  $ r_i={\rm ord}(V_{i\B})$ for $i=1,2$ and  without loss of generality, assume that $r_2\geq r_1$ and that $V_{2\B}$ is nontrivial if one of $V_{1\B}$ and $V_{2\B}$ is a nontrivial $\B$-module.
Assume that $$\phi: \mathcal{M}\big(V_1,\mu_1,\Omega(\lambda_1,\alpha_1)\big)\rightarrow \mathcal{M}\big(V_2,\mu_2,\Omega(\lambda_2,\alpha_2)\big)$$  is an  isomorphism of $\L$-modules.
Take any $0\neq w\in V_{1\B}$  and assume that $\phi(w\otimes 1)=\sum_{i=0}^pu_i\otimes \partial^i\in \mathcal{M}\big(V_2,\mu_2,\Omega(\lambda_2,\alpha_2)\big)$ with $u_p\neq0$. Then we have $\phi(w\otimes f(\partial))=\sum_{i=0}^pu_i\otimes f(\partial)\partial^{i}$ for $f(\partial)\in\C[\partial]$ by repeatedly using the action of $L_0$.   It follows from this and $\phi\big(L_m(w\otimes 1)\big)=L_m\phi(w\otimes 1)$ that
\begin{eqnarray}\label{h-eq1}
&&\sum_{i} u_i\otimes (\partial-m\alpha_1)\partial^i+\phi\Big(\big((\mu_1^m\sum_{j=0}^{r_1+1}\frac{m^j}{j!}L_{j-1}-L_{-1})w\big)\otimes1\Big)\\
&=&\!\!\!\big(\frac{\lambda_2}{\lambda_1}\big)^m\sum_{i}\Big\{ u_i\otimes (\partial-m\alpha_2)(\partial-m)^i+(\mu_2^m\sum_{j=0}^{r^{(i)}+1}\frac{m^j}{j!}L_{j-1}-L_{-1})u_i\otimes (\partial-m)^i\Big\},\nonumber
\end{eqnarray}
where $r^{(i)}={\rm ord}(u_i)$.  By the definition of $r_2$ we see that $r^{(p)}\geq r_2$. Now by Proposition \ref{pro13.3} we see that $p=0$ and $r_1=r_2=r^{(0)}$. This allows us to define an injective linear map $\varphi:V_{1\B}\rightarrow V_{2\B}$ such that $\phi(w\otimes 1)=\varphi (w)\otimes 1$ for $w\in V_{1\B}$. Whence \eqref{h-eq1} is simplified as
\begin{eqnarray}\label{h-eq2}
&&\varphi(w)\otimes (\partial-m\alpha_1)+\mu_1^m\varphi(\sum_{j=1}^{r_1+1}\frac{m^j}{j!}L_{j-1}w)\otimes1+(\mu_1^m-1)\phi\big((L_{-1}w)\otimes1\big)\\
&=&\!\!\!\big(\frac{\lambda_2}{\lambda_1}\big)^m\Big\{ \varphi(w)\otimes (\partial-m\alpha_2)+\mu_2^m\sum_{j=1}^{r_1+1}\frac{m^j}{j!}\big(L_{j-1}\varphi(w)\big)\otimes1+(\mu_2^m-1)\big(L_{-1}\varphi(w)\big)\otimes 1\Big\}.\nonumber
\end{eqnarray}Consider first that $V_{2\B}$ is a trivial $\B$-module, then so is $V_{1\B}$ by our assumption at the beginning of this proof. In particular, $V_{1\B}\cong V_{2\B}$, and the second terms on both sides of \eqref{h-eq2} vanish. It then follows immediately from  Proposition \ref{pro13.3} that $(\mu_1, \lambda_1, \alpha_1)=(\mu_2, \lambda_2, \alpha_2)$

Now assume that $V_{2\B}$ is nontrivial. Using Lemma \ref{lemm3.1} and comparing the maximal order of the first factors of elements in \eqref{h-eq2}  give $$\lambda_2^m(\mu_2^m-1)=\lambda_1^m(\mu_1^m-1)c_w,\  $$  where $c_w\in\C^*$ satisfies $\phi\big((L_{-1}w)\otimes1\big)\equiv c_w\big(L_{-1}\varphi(w)\big)\otimes 1\ ({\rm mod}\ V_{2\B}\otimes\C[\partial]).$ Then by Proposition \ref{pro13.3} we can conclude the following two cases:

\begin{case}
$(\mu_2,\lambda_2)=(\mu_1,\lambda_1)$ and $c_w=1$.
\end{case}
But inserting these into \eqref{h-eq2} and using Proposition \ref{pro13.3} yield $\alpha_2=\alpha_1$ and $\varphi(L_i w)=L_i\varphi(w)$ for all $i\geq 0.$ Thus, $$ (\mu_1, \lambda_1,\alpha_1)=(\mu_2, \lambda_2,\alpha_2)\ {\rm and}\ \varphi\ {\rm is\ a\ homomorphism\ of}\ \H{\text-}{\rm modules}.$$
Whence the injectivity of $\varphi$ and the simplicity of $V_{1\B}, V_{2\B}$ (see Lemma \ref{lemm3.1}(1)) imply that $\varphi$ is an isomorphism, which in turn  implies  $V_1\cong V_2$ by Lemma \ref{lemm3.1}(4). This is (a).

\begin{case}\label{c-a-s-e}
$\lambda_2=\mu_1\lambda_1$ and $c_w=-1$.
\end{case}
While inserting $\lambda_2=\mu_1\lambda_1$ into \eqref{h-eq2} and using Proposition \ref{pro13.3} yield that $\lambda_1=\mu_2\lambda_2$ and that $V_{1\B}, V_{2\B}$ are highest weight $\B$-module with highest weights $-\alpha_2, -\alpha_1,$ respectively. Note that this is (b).
\end{proof}

\section{New simple modules $\mathcal{M}\big(V,\mu,\Omega(\lambda,\alpha)\big)$}
Let us first recall  simple non-weight $\L$-modules from \cite{H,LGW,LLZ,TZ1,MZ,LZ2}.
For any $\lambda,\alpha\in\C^*$ and $h(t)\in\C[t]$ such that ${\rm deg}\, h(t)=1$, $\Phi(\lambda,\alpha,h):=\C[s,t]$ carries the structure of an $\L$-module:
$L_mf(s,t)=\sum_{j=0}^\infty \lambda^m(-m)^j S^jf(s,t)$ and $Cf(s,t)=0$, where $$S^j=\frac s{j!}\partial_s^j-\frac 1{(j-1)!}\partial_s^{j-1}\big(t(\eta-\partial_t)+h(\alpha)\big)-\frac{1}{(j-2)!}\partial_s^{j-2}\alpha(\eta-\partial_t)\quad{\rm for}\ j\in\Z_+,$$ $\partial_t=\frac{\partial}{\partial t}, \partial_s=\frac{\partial}{\partial s}$ and $\eta=\frac{h(t)-h(\alpha)}{t-\alpha}\in\C^*$. Here we decree $\big(\frac\partial {\partial s}\big)^{-1}=0$,  $k!=1$ for $k<0$ and $\binom{i}{j}=0$  for  $j> i$ or $j<0$.

Let $V$ be  a simple  $\L$-module for which there exists $R_V\in\Z_+$ such that $L_m$ for all $m\geq R_V$ are locally nilpotent on $V$. In fact, such kind of simple $\L$-modules were already classified in \cite{MZ}.
It was respectively shown in \cite{H,LGW,TZ1} that the tensor products  $\L$-modules   $\otimes_{i=1}^n\Phi(\lambda_i,\alpha_i,h_i(t))\otimes V$, $\otimes_{i=1}^n\Phi(\lambda_i,\alpha_i, h_i(t))\otimes\otimes_{i=1}^m\Omega(\lambda_i,\alpha_i)\otimes M$ and  $\otimes_{i=1}^{m}\Omega(\lambda_i,\alpha_i)\otimes V$ are  simple if $\lambda_1,\cdots,\lambda_n,\mu_1,\cdots,\mu_m$ are pairwise distinct.
%These are a generalization of  tensor product $\L$-modules of $\Omega(\lambda,\alpha)$ with $\mathrm{Ind}(M)$ or the highest weight modules (see \cite{TZ}).

%For any $\xi\in\C$, we denote
%$$\L_+^\xi=\bigoplus_{i\in\N}\C(t-\xi)^{i+1}\partial,\ \L_0^\xi=\C(t-\xi)\partial,\ \L_-^\xi=\bigoplus_{i\in\N}\C t^{-i+1}\partial.$$
%The {\em highest-weight-like Verma } $\L$-module is defined by
%$$M(\xi,\lambda)=U(\L)\otimes_{U(\L_+^\xi+\L_0^\xi)}\C v,$$
%where $\L_+^\xi\cdot v=0$ and $(t-\xi)\partial\cdot v=\lambda v$  for some $\lambda\in\C$.
%Denote by $R_a$  the subalgebra of the rational field in the indeterminant $t$, which  generated by $t,(t-a_0)^{-1},\ldots,(t-a_n)^{-1}$.
%The fraction module $V(a,\alpha,\beta)$ of $\L$-action is defined as follows
%$$(f\partial)\circ(gz)=f\partial(gz)+\beta gz\partial(f),$$
%where $f \in\C[t^{\pm1}], g\in R_a,z=t^{\alpha_0}(t-\alpha_1)^{\alpha_1}\cdots(t-\alpha_n)^{\alpha_n}$. The simplicity of $\L$-module
%$V(a,\alpha,\beta)$ was determined in \cite{GLZ}.

Let $b\in\C$ and $A$ be a simple module over the associative algebra $\mathcal{K}=\C[t^{\pm1}, t\frac{d}{dt}].$  The  action of $\L$  on $A_b:=A$ is given by
$$Cv=0, L_nv=(t^{n+1}\frac{d}{dt}+nbt^n)v\quad \mbox{for}\ n\in\Z,v\in A.$$
It was proved in \cite{LZ2} that $A_b$ is a simple $\L$-module if and only if one of the following conditions holds: (1) $b\neq 0\ \mbox{or}\ 1$; (2) $b=1$ and $t\frac{d}{dt} A =A$; (3) $b=0$ and $A$ is not isomorphic to the natural $\mathcal{K}$-module $\C[t,t^{-1}]$.
%\begin{itemize}\lineskip0pt\parskip-1pt \item[\rm (1)] $b\neq 0\ \mbox{or}\ 1$; \item[\rm(2)] $b=1$ and $t\frac{d}{dt} A =A$; \item[\rm(3)]  $b=0$ and $A$ is not isomorphic to the natural $\mathcal{K}$ module $\C[t,t^{-1}]$.\end{itemize}

For any  $r\in\Z_+,$ denote  $\B_r$ to be the quotient algebra of $\B$ by $\B^{(r+1)}=\mathrm{span}\{L_i\mid  i\geq r+1\}$. The classification of simple modules over $\B_i$  for $i=1,2$ were respectively obtained in \cite{B} and   \cite{MZ},  and
 remains unsolved for $r\geq3$. Let $V$  be a $\B_{r}$-module.
For any $\gamma(t)=\sum_{i}c_it^{i}\in\C[t,t^{-1}]$,
define the  action of $\L$ on $V\otimes \C[t,t^{-1}]$ as follows
\begin{eqnarray*}\label{l6.1}
&L_m(v\otimes t^n)=\Big(a+n+\sum_{i=0}^r\frac{m^{i+1}}{(i+1)!} L_i\Big)v\otimes t^{m+n}+\sum_{i}c_iv\otimes  t^{n+i},\\
&C(v\otimes t^n)=0 \quad {\rm for}\ m,n\in\Z, v\in V.
 \end{eqnarray*}
Then $V\otimes \C[t,t^{-1}]$ carries the structure of an $\L$-module under the  above given actions, which is denoted by $\widetilde{\mathcal M}(V,\gamma(t))$. Note from \cite{LLZ} that  $\widetilde{\mathcal M}(V,\gamma(t))$ is a weight $\L$-module if and only if $\gamma(t)\in \C$ and also that
the $\L$-module $\widetilde{\mathcal M}(V,\gamma(t))$ for $\gamma(t)\in\C[t,t^{-1}]$ is simple if and only if $V$ is simple (see also \cite{CHS}).

\begin{prop}\label{5.1}
Let $1\neq \mu,\alpha,\alpha_i\in\C^\ast,\lambda,\lambda_i\in\C$ with $\lambda_i$ pairwise distinct for $i=1,\ldots,n,$   $M$ be a simple $\L$-module for which there exists $R_M\in\Z_+$ such that $L_m$ for all $m\geq R_M$ are locally nilpotent on $M$ and $V$   a  simple $\H$-module in $\mathcal C.$
Then $$\mathcal{M}\big(V,\mu,\Omega(\lambda,\alpha)\big)\cong \otimes_{i=1}^n\Omega(\lambda_i,\alpha_i)\otimes M$$ if and only if $$M\ {\rm is\ a \ trivial\ } \L\text-{\rm module}\  {\rm and}\ (n,\lambda,\lambda\mu,\alpha,\beta)=(2, \lambda_{\sigma1},\lambda_{\sigma2},\alpha_{\sigma1},-\alpha_{\sigma2})\ {\rm for\ some}\ \sigma\in S_2,$$ where $\beta$
is the highest weight of $V_{\B}$.
\end{prop}
\begin{proof} Let $r\geq-1$ be the maximal integer such that $L_r$ is injective on $V_\B$. Assume  $\lambda=1$ for convenience.
Suppose that $$\phi: \mathcal{M}\big(V,\mu,\Omega(1,\alpha)\big)\rightarrow \otimes_{i=1}^n\Omega(\lambda_i,\alpha_i)\otimes M$$  is an  isomorphism of $\L$-modules.
  Take
  any $0\neq v\in V_\B$ and assume that
  \begin{eqnarray}\label{5.11}
  \phi(v\otimes1)=\sum_{{\bf k}=(k_1,\ldots,k_n)\in I} \partial^{k_1}\otimes\cdots\otimes\partial^{k_n}\otimes v_{{\bf k}}
  \end{eqnarray}for some finite subset $I$ of $\Z_+^n$ such that $\{\partial^{k_1}\otimes \cdots\otimes \partial^{k_n}\otimes v_{\bf k}\mid {\bf k}\in I\}$ is linearly independent. Choose $p$ large enough so that $L_m v_{\bf k}=0$ for all $m\geq p$ and ${\bf k}\in I$. It follows from  $\phi\big(L_m(v\otimes1)\big)=L_m\sum_{{\bf k}\in I} \partial^{k_1}\otimes\cdots\otimes\partial^{k_n}\otimes v_{\bf k}$ that
\begin{eqnarray}\label{sect5-ad-a1}
&&\phi\big(v\otimes (\partial-m\alpha)+\mu^m\sum_{j=0}^{r+1}\frac{m^j}{j!}L_{j-1}v\otimes1-L_{-1}v\otimes1\big)\\
&=&\sum_{{\bf k}\in I}\sum_{i=0}^n \partial^{k_1}\otimes\cdots\otimes \partial^{k_{i-1}}\otimes\lambda_i^m(\partial-m\alpha_i)(\partial-m)^{k_i}\otimes\partial^{k_{i+1}}\cdots\otimes\partial^{k_n}\otimes v_{{\bf k}}\ {\rm for}\ m\geq p.\nonumber
\end{eqnarray}
By Proposition \ref{pro13.3}, we know that   $(n,\lambda_{\sigma1}, \lambda_{\sigma2},\alpha_{\sigma1})=(2,1,\mu,\alpha)$ for some $\sigma\in S_2.$  Without loss of generality, we assume that $\sigma=1$, namely, $(\lambda_{1}, \lambda_{2},\alpha_{1})=(1,\mu,\alpha)$. Note also from \eqref{sect5-ad-a1}  on one hand that $r\neq-1$, since otherwise $\mu^mm^i$ for some $i\geq1$ would only be the coefficient of some nonzero term on the right hand side; and on the other hand that $k_1=0,k_2\le r$ for all ${\bf k}=(k_1,k_2)\in I$ and $k_2=r$ holds only  for one ${\bf k}$. Whence \eqref{5.11} can be written as $\phi(v\otimes1)=\sum_{i=0}^r 1\otimes\partial^{i}\otimes v_{i}.$

Consider first $r\geq1$.   It follows from  $\phi\big(L_m^2(v\otimes1)\big)=L_m^2\sum_{i=0}^r(1\otimes \partial^i\otimes v_{i})$ that
 \begin{eqnarray*}\label{5.2}
&&\phi\Big\{v\otimes(\partial-m\alpha)(\partial-m\alpha-m)+
\big((\mu^{m}e^{mt}-1)\frac{d}{dt}\big)v\otimes (\partial-m\alpha-m) +\nonumber
\\&&~~~~\big((\mu^{m}e^{mt}-1)\frac{d}{dt}\big)v\otimes (\partial-m\alpha) + \big((\mu^{m}e^{mt}-1)\frac{d}{dt}\big)\big((\mu^{m}e^{mt}-1)\frac{d}{dt}\big)v\otimes1 \Big\}\nonumber
\\&=&\sum_{i=0}^r\Big\{(\partial-m\alpha)(\partial-m\alpha-m)\otimes\partial^i\otimes v_i+2(\partial-m\alpha)\otimes\mu^m(\partial-m\alpha_2)(\partial-m)^r\otimes v_i\nonumber
\\&&~~~~~~\ + 1\otimes\mu^{2m}(\partial-m\alpha_2)(\partial-m\alpha_2-m)(\partial-2m)^i\otimes v_i\Big\},
 \end{eqnarray*}which gives $\phi(L_r^2v\otimes1)=0$ by comparing the coefficient of $\mu^{2m}m^{2r+2}.$ This contradicts the injectivity of $\phi$, since $L_r^2v\otimes 1\neq0$. Thus,  $\mathcal{M}\big(V,\mu,\Omega(\lambda,\alpha)\big)\ncong \otimes_{i=1}^n\Omega(\lambda_i,\alpha_i)\otimes M$ in this case.

The remaining case is  $V_{\B}$ being a nontrivial highest weight module. Assume that its  highest weight is $\beta\in\C^*$. Now it is not hard to deduce from $\phi\big(L_m(v\otimes1)\big)=L_m(1\otimes1\otimes v_0)$ that $\beta=-\alpha_2$ and $M$ is a trivial $\L$-module.

 Conversely, suppose that $M$ is a trivial $\L$-module, $V_\B$ is a highest weight module with highest weight $-\alpha_2$ and $(n,1,\mu,\alpha)=(2, \lambda_{1},\lambda_{2},\alpha_{1})$ (this is the the case $\sigma=1$). Then   following the proof of Lemma \ref{lemm4.1} one can check that the linear map $\varphi: \mathcal{M}\big(V,\mu,\Omega(1,\alpha)\big)\rightarrow \Omega(1,\alpha)\otimes\Omega(\mu,-\beta)\otimes \C$ sending $L_{-1}^i v\otimes \partial^j$ to $\sum_{p=0}^j\binom{j}{p}\partial^{j-p}\otimes \partial^{i+p}\otimes 1$ for any $i,j\in\Z_+$ is an isomorphism of $\L$-module, where $v$ is the highest weight vector of $V_\B$.
 \end{proof}

\begin{prop}
Let $ \lambda,\alpha,1\neq \mu\in\C^*$ and  $V$ be a simple $\H$-module in $\mathcal C$. Then  $\mathcal{M}\big(V,\mu,$ $\Omega(\lambda,\alpha)\big)$ is not isomorphic  to any of the following simple $\L$-modules: $$M, \otimes_{i=1}^n\Phi\big(\lambda_i,\alpha_i, h_i(t)\big)\otimes M, \otimes_{i=1}^n\Phi\big(\lambda_i,\alpha_i, h_i(t)\big)\otimes\otimes_{i=1}^m\Omega(\mu_i,\alpha_i)\otimes M, A_b, \widetilde{\mathcal M}\big(W,\gamma(t)\big),$$
where  $m,n\geq1,\lambda_i,\alpha_i\in\C^*, b\in\C,$ $\gamma(t),h_i(t)\in\C[t]$ with ${\rm deg}\ h_i(t)=1,$ $\lambda_1,\cdots,\lambda_n,\mu_1,\cdots,$ $\mu_m$ being pairwise distinct$,$ $M$ is a simple $\L$-module for which there exists $R_M\in\Z_+$ such that  $L_m$  is locally nilpotent on $M$ for all $m\geq R_M,$  $W$ is a  simple $\B$-module.
\end{prop}
\begin{proof} For convenience, we assume that $\lambda=1$. Since for any large enough $m,$  $L_m$ is  locally nilpotent on  $M$ but not on $\mathcal{M}\big(V,\mu,\Omega(\lambda,\alpha)\big)$,  $ M\ncong\mathcal{M}\big(V,\mu,$ $\Omega(\lambda,\alpha)\big).$ Suppose that $$\phi:   \mathcal{M}\big(V,\mu, \Omega(1,\alpha)\big)\rightarrow\otimes_{i=1}^n\Phi\big(\lambda_i,\alpha_i,h_i(t)\big)\otimes M$$ is an isomorphism of $\L$-modules.   Take $0\neq u\in V$   and assume  $\phi(u\otimes 1)=\sum_{i\in I}f_{1i}\otimes f_{2i}\otimes\cdots\otimes f_{ni}\otimes u_i$.  Choose $m$ to be large enough so that $L_m u_i=0$ for all $i\in I$. Then it follows from $\phi\big(L_m(u\otimes 1)\big)=L_m\phi(u\otimes 1)$ that
\begin{eqnarray*}
&&\phi\Big(u\otimes (\partial-m\alpha)+\big((\mu^me^{mt}-1)\frac{d}{dt}\big)u\otimes 1\Big)\\
&=&\sum_{i\in I}\sum_{k=1}^n\sum_{j=0}^\infty\lambda_k^m(-m)^j\underbrace{f_{1i}\otimes\cdots\otimes}_{k-1}S^jf_{ki}\otimes \cdots\otimes u_i
\end{eqnarray*}Then in view of   Proposition \ref{pro13.3} we have $n=2$ and may assume that $\lambda_1=1$ and $\lambda_2=\mu$; moreover, $\sum_{i\in I} S^jf_{1i}\otimes f_{2i}\otimes u_i=0$ for all $j\geq2$. But we see that $\sum_{i\in I} S^2f_{1i}\otimes f_{2i}\otimes u_i\neq0$  by choosing $f_{1i}$  to be of form $t^{n_i}$ for $i\in I$,   a contradiction. This shows $  \mathcal{M}\big(V,\mu, \Omega(1,\alpha)\big)\ncong\otimes_{i=1}^n\Phi\big(\lambda_i,\alpha_i,h_i(t)\big)\otimes M.$ Similarly, one has $$\mathcal{M}\big(V,\mu, \Omega(1,\alpha)\big)\ncong\otimes_{i=1}^n\Phi\big(\lambda_i,\alpha_i, h_i(t)\big)\otimes\otimes_{i=1}^m\Omega(\lambda_i,\alpha_i)\otimes M.$$
And  the non-isomorphisms    $\mathcal{M}\big(V,\mu,$ $\Omega(1,\alpha)\big)\ncong A_b$  and   $\mathcal{M}\big(V,\mu,\Omega(1,\alpha)\big)\ncong\widetilde{\mathcal M}\big(W,\gamma(t)\big)$ follows immediately by using Proposition \ref{pro13.3}, completing the proof.
\end{proof}
Now we can conclude this section with the following corollary.
\begin{coro}
Let $ \lambda,\alpha,1\neq \mu\in\C^*$ and  $V$ be a simple $\H$-module in $\mathcal C$. Suppose that $V$ is not a highest weight $\mathfrak a$-module whose highest weight is nonzero. Then $\mathcal{M}\big(V,\mu,$ $\Omega(\lambda,\alpha)\big)$ is not isomorphic to any simple $\L$-module in \cite{H,LGW,LLZ,TZ1,MZ,LZ2}.
\end{coro}

\end{document}